\documentclass[12pt,reqno]{amsart}
\usepackage[utf8]{inputenc}
\usepackage{amsmath,amsthm,amsfonts,amssymb,commath,mathtools}
\usepackage[backref,colorlinks,citecolor=blue,bookmarks=true]{hyperref}
\usepackage{float}
\usepackage{tikz}
\usetikzlibrary{calc}

\usepackage{pgf}

\newcommand{\newauthor}[3]{
\author{#1}
\thanks{\texttt{#2} #3}
        }
\newcommand{\stepref}[1]{
	\hyperref[#1]{step \ref*{#1}}
    }        
\title{Maximal Planar Subgraphs of Fixed Girth in Random Graphs}
\newauthor{Manuel Fern\'{a}ndez}{manuelf@andrew.cmu.edu}{School of Computer Science, Carnegie Mellon University, Pittsburgh, Pennsylvania, 15213, USA}
\newauthor{Nicholas Sieger}{nsieger@andrew.cmu.edu}{Department of Mathematical Sciences, Carnegie Mellon University, Pittsburgh, Pennsylvania, 15213, USA}
\newauthor{Michael Tait}{mtait@.cmu.edu}{Department of Mathematical Sciences, Carnegie Mellon University, Pittsburgh, Pennsylvania, 15213, USA. The third author's research is supported by NSF grant DMS-1606350.}

\newcommand{\addverts}[3]{
	\foreach \coord in #1 {
		\node[circle,fill=#2!25,minimum size=17pt,inner sep=0pt] at (\coord) {$#3$};
	}	
}

\newcommand{\bisect}[4]{
	\coordinate (#3) at ($(#1)!0.5!(#2)$);
	\draw[#4] (#1) -- (#3) -- (#2);
	}

\newcommand{\addpath}[4]{
	\draw[dotted,#4,line width = 1mm] (#1) -- (#2)
		node[midway,above,outer sep = 7pt,text = black] {$$#3$$};
}

\theoremstyle{plain}
\newtheorem{thm}{Theorem}[section]
\newtheorem{lem}{Lemma}[section]

\theoremstyle{definition}

\newtheorem{clm}{Claim}[section]

\renewcommand{\Pr}{\mathop{\bf Pr\/}}          

\newcommand{\Ex}{\mathop{\bf E\/}}

\begin{document}
\maketitle
\begin{abstract}
 In 1991, Bollob\'{a}s and Frieze showed that the threshold for $G_{n,p}$ to contain a spanning maximal planar subgraph is very close to $p = n^{-1/3}$. In this paper, we compute similar threshold ranges for $G_{n,p}$ to contain a maximal bipartite planar subgraph and for $G_{n,p}$ to contain a maximal planar subgraph of fixed girth $g$.
\end{abstract}
\section{Introduction}
In the field of random graphs, determining the threshold $p$ for a random graph to have a certain graph property $\mathcal{G}$ with probability tending to $1$ (subsequently, w.h.p.) is a fundamental problem. The threshold $f(n)$ for a graph property is a function of $n$ for which if $p = o\left(f(n)\right)$ then w.h.p. $G_{n,p} \not \in \mathcal{G}$ and if $f(n) = o(p)$ then w.h.p. $G_{n,p} \in \mathcal{G}$. Although not every graph property has a threshold in a random graph, it is a well-known fact that every monotonic graph property does \cite{FK1996}. One such property is that of a random graph containing a spanning maximal planar subgraph. In 1991, Bollob\'{a}s and Frieze \cite{BF91} determined that the threshold for this property lies in the interval $\left(c_1\left(\frac{1}{n}\right)^{1/3},c_2\left(\frac{\log n}{n}\right)^{1/3}\right)$, for some constants $c_1,c_2 \in \mathbb{R}$. It is currently an open problem to determine the exact threshold. However, the above result motivates the other questions. In particular, a natural generalization of this result would be to determine relatively small intervals for the threshold for a random graph to contain a maximal planar subgraph, but a with a given girth $g$. In this respect, the above result gives such an interval for girth $g = 3$. 

Euler's formula implies that if a planar graph on $n$ vertices has girth $g$, then $e(G) \leq \frac{g}{g-2}(n-2)$. When equality occurs we call such a graph {\em maximal}. For equality to occur, if $g$ is odd then $n \equiv 2 \pmod{g-2}$ and if $g$ is even then $n \equiv 2 \pmod{(g-2)/2}$. It is easy to construct graphs that show this necessary condition is also sufficient for the existence of a maximal planar graph of girth $g$. 
\par
In this paper, we show that the threshold for a random graph to contain a maximal planar subgraph of girth $g$ lies in a small interval when the number of vertices satisfies the necessary divisibility conditions. We also prove a corresponding theorem for when a random graph contains a maximal bipartite planar subgraph.
\begin{thm}\label{bipartite}
Let $\mathcal{G}$ be the graph property that a graph contains a spanning maximal planar bipartite subgraph. Then the threshold for $G_{n,p} \in \mathcal{G}$ is contained in the interval $(c_1n^{-1/2},c_2(log ~ n)^{1/2}n^{-1/2})$ for some constants $c_1,c_2 \in \mathbb{R}.$ \end{thm}
\begin{thm}\label{evencase}
Let $\mathcal{G}$ be the graph property that a graph contains a spanning maximal planar subgraph of girth $g = 2k$ for $k \geq 3$. Then the threshold for $G_{\frac{(g-2)}{2}n + 2, p} \in \mathcal{G}$ is contained in the interval $\left(\left(\frac{c_1}{n}\right)^{(g-2)/g},\left(\frac{c_2(\log n)^{(g+2)/g}}{n^{(g-2)/g}}\right)\right)$ for some constants $c_1,c_2 \in \mathbb{R}.$
\end{thm}
\begin{thm}\label{oddcase}
Let $\mathcal{G}$ be the graph property that a graph contains a spanning maximal planar subgraph of girth $g = 2k + 1$ for $k \geq 3$. Then the threshold for $G_{(k-1)n+2, p} \in \mathcal{G}$ is contained in the interval $\left(\left(\frac{c_1}{n}\right)^{(g-2)/g},\left(\frac{c_2(\log n)^{1/k}}{n^{(g-4)/(g-2)}}\right)\right)$ for some constants $c_1,c_2 \in \mathbb{R}.$
\end{thm}
The rest of the paper will proceed as follows: First we introduce the necessary theorems and standardized notation. Next, we prove the result for bipartite subgraphs before proving the results for larger girth. 
\section{Preliminaries}
Let $G_{n,p}$ be the Erd\H{o}s-Renyi random graph model. All of our asymptotic notation will be as $n\to\infty$.

The proof the lower thresholds in Theorems \ref{bipartite}, \ref{evencase}, and \ref{oddcase} all follow from a simple lemma.
\begin{lem}\label{lowthres}
	Let $g$ be fixed. If $p = \left(\frac{c_1}{n}\right)^{(g - 2)/g}$ for a small enough constant $c_1$, then $G_{n,p}$ contains no spanning girth $g$ maximal planar subgraph.
\end{lem}
\begin{proof}
	By \cite{GN09}, the number of spanning planar graphs is asymptotic to $b\cdot n^{-7/2} \gamma^n n!$ for explicit constants $b$ and $\gamma$. Therefore, for a fixed $g$ we can clearly say that the number of maximal planar graphs of girth $g$ on $n$ vertices is at most $(cn)^n$ for some constant $c$. Let $p = \left(\frac{c_1}{n}\right)^{(g-2)/g}$ where $c_1 < \max\{1/c, 1\}$. By the union bound, for each graph property in Theorems \ref{bipartite}, \ref{evencase}, and \ref{oddcase} we have the following.
    \begin{align*}
    	\Pr[G_{n,p}\in \mathcal{G}] &\leq (cn)^np^{\frac{g(n - 2)}{g - 2}}\\
        &< (cn)^n\left(\frac{c_1}{n}\right)^{n-2}\\
        &\to 0
    \end{align*}
\end{proof}

In proving the upper thresholds, we need to show that if $p$ is large enough, then we can find a maximal planar girth $g$ subgraph with high probability. To do this we will show that an explicitly constructed subgraph can be found. We will follow the ideas of \cite{BF91} by inserting vertices into a specific construction. Where Bollab\'as and Frieze inserted edges in each round, we instead insert paths. To do this, we will make frequent use of the following lemma from \cite{Kri10}. Note that the values $C_1,C_2$ follow from the results in \cite{JKV08}.
\begin{lem}\label{pathlemma}
	Let $k \geq 3$ be a fixed integer, and let $G$ be distributed as $G_{(k + 1)n_0,p}$.Let $S = \{s_1,\dots,s_{n_0}\},T = \{t_1,\dots,t_{n_0}\}$ be disjoints subsets of $[(k + 1)n_0]$.Finally, let $C_1 > 0$ be a fixed constant. Then there exists $C_2 > 0$ such that if
    \[
    	p \geq C_2\left(\frac{\ln n_0}{n_0^{k - 1}}\right)^{1/k}
    \] then with probability $1 - n_0^{-C_1}$, $G$ contains a family $\{P_i\}_{i = 1}^{n_0}$ of vertex disjoint paths, where each $P_i$ is a path of length $k$ connecting $s_i$ to $t_i$.
\end{lem}
Furthermore, Bollab\'as and Frieze complete their proof by inserting all remaining vertices into faces of their construction. We will use a similar technique to complete our constructions, motivating the following lemma.
\begin{lem}\label{finalmatching}
Suppose $p = \frac{C\log n}{n}$ for some constant $C \geq 2.$ Let $X,Y$ be non-empty vertex sets such that $X \cap Y = \emptyset, |Y| \geq 3n/4, |X| \leq n/4.$ and an edge connects $x$ and $y$ with probability p, for every $x \in X, y \in Y.$ Then the probability there is a maximal matching between $X$ and $Y$ is at least $1 - \Omega\left(\frac{1}{n \log n}\right)$.  
\end{lem}
\begin{proof}
The proof follows by a second moment argument. Let Z denote the number of maximal matchings between $X$ and $Y$ in $G_{n,p}$ Then 
\[
\Ex[Z] = \frac{|X|!\cdot |Y|!}{(|Y| - |X|)!}\frac{1}{|X|!}p^{|X|} = \frac{|Y|!}{(|Y| - |X|)!}p^{|X|}
\]
\begin{align*}
\Ex[Z^2] &\leq \sum_{k = 0}^{|X|} \frac{|X|! \cdot |Y|!}{(|Y| - k)!(|X| - k)!k!}p^k \left(\frac{(|Y| - k)!}{(|Y| - (|X| - k))!}p^{|X| - k} \right)^2\\
&= |X|!|Y|!p^{2|X|}\sum_{k = 0}^{|X|}\frac{(|Y| - k)!}{((|Y| - (|X| - k))!)^2(|X| - k)!}\frac{p^{-k}}{k!}
\end{align*}

Now note that since $|X| \leq n/4, |Y| \geq 3n/4$, we have that 
\begin{align*}
\frac{((|Y| - (|X| - (k+1)))!)(|X| - (k+1))!}{((|Y| - (|X| - k))!)(|X| - k)!} &= \frac{|Y| + 1 - |X| + k}{(|X| - k)} \\ & \geq \frac{|Y| + 1 - |X|}{|X|} \geq \frac{n/2}{n/4} = 2
\end{align*}
for all $k < |X|.$ It follows that 
\begin{align*}
\Ex[Z^2] &\leq |X|!|Y|!p^{2|X|}\sum_{k = 0}^{|X|}\frac{(|Y| - k)!}{((|Y| - (|X| - k))!)}\frac{(2p)^{-k}}{k!} \\ &\leq \frac{(|Y|!)^2}{(|Y| - |X|)!}p^{2|X|}\sum_{k = 0}^{|X|}\frac{(|Y| - k)!}{((|Y| - (|X| - k))!)}\frac{(2p)^{-k}}{k!}
\end{align*}
Now note that 
\[\frac{(|Y| - (k + 1))!(2p)^{-(k+1)}}{(|Y| - k)!(2p)^{-k}} \leq (2p)^{-1}/(n/2) = n/(Cn \log n) \leq (C \log n)^{-1}
\]
for all $k < |X|.$ Thus the above expression is at most
$$
\frac{(|Y|!)^2}{(|Y| - |X|)!}p^{2|X|}\sum_{k = 0}^{|X|}\frac{1}{((|Y| - (|X| - k))!)}\frac{(C \log n)^{-k}}{k!}.
$$
Lastly note that $\frac{(|Y| - (|X| - (k + 1)))!}{(|Y| - (|X| - k))!} \geq n/2$ and so 
\[
\frac{1}{(|Y| - (|X|-k))!} \leq \frac{1}{(|Y|-|X|)!}\left(\frac{2}{n}\right)^k.
\]
Since $C \geq 2$, it follows that the above expression is at most
\begin{align*}
\frac{(|Y|!)^2}{((|Y| - |X|)!)^2}p^{2|X|}\sum_{k = 0}^{|X|}\frac{1}{(n\log n)^{-k}k!} &\leq \frac{(|Y|!)^2}{((|Y| - |X|)!)^2}p^{2|X|}\sum_{k = 0}^{|X|}\frac{1}{(n\log n)^{-k}} \\
&= (\Ex[Z])^2\sum_{k = 0}^{|X|}\frac{1}{(n\log n)^{-k}}\\
&= (\Ex[Z])^2\left(1 + \frac{1}{n \log n} + O\left(\frac{1}{n^2\log n^2}\right)\right)
\end{align*}
Finally, by Paley-Zygmund we have that 
\[\Pr(Z > 0) \geq \frac{\Ex[Z]^2}{\Ex[Z^2]} \geq \frac{n \log n}{n \log n + 1 + o(1)} \geq 1 - \frac{1+o(1)}{n \log n}
\]
\end{proof}    
Additionally, we must check that our constructed subgraph is maximal. From our construction and the following claim this will be clear.

\begin{clm}\label{sizeofmaxplanarG}
If $G$ is a planar subgraph of girth $g$, and every face of $G$ has size $g$, then $G$ is maximal.
\end{clm}

\section{The Bipartite Case} 
The proof of the upper bound is motivated by a construction from \cite{BF91}.
Throughout this section assume $p = \frac{C(\log n)^{1/2}}{n^{1/2}}$ where $C$ is a large enough constant.

Let $G_{n,p}$ be the random graph and let $E$ be its edge set. Consider the following construction:

\begin{enumerate}
\item \label{bipartsplitedges}
	Define the 8 independent random edges sets $E_1,E_2,\cdots,E_8,$ each distributed as $G_{n,p'}$ where $(1-p) = (1-p')^8$ and
$$
\bigcup_{i = 1}^8 E_i = E.
$$
\item \label{bipartinitialcycle}
	Construct a 4-cycle using the edges from $E_1$.
\item \label{bipartmainconstr}
	Consider the diagram below. 
	\begin{figure}
		\centering
		\begin{tikzpicture}
		\tikzstyle{vertex}=[circle,fill=black!25,minimum size=17pt,inner sep=0pt];

		\coordinate (a) at (-5,5);
		\coordinate (b) at (5,5);
		\coordinate (c) at (5,-5);
		\coordinate (d) at (-5,-5);
		\draw[blue] (a) -- (b) -- (c) -- (d) -- (a);	

		\bisect{a}{c}{M}{red};
		\bisect{b}{M}{f}{green};
		\bisect{M}{d}{g}{green};
		
		\bisect{a}{f}{h}{yellow};
		\bisect{c}{f}{i}{yellow};
		\bisect{c}{g}{j}{yellow};
		\bisect{a}{g}{k}{yellow};
		
		\bisect{h}{M}{l}{purple};
		\bisect{i}{M}{m}{purple};
		\bisect{j}{M}{n}{purple};
		\bisect{k}{M}{o}{purple};
		
		\bisect{l}{f}{p}{cyan};
		\bisect{m}{f}{q}{cyan};
		\bisect{o}{g}{r}{cyan};
		\bisect{n}{g}{s}{cyan};
		
		\fill [blue!10] (h) -- (l) -- (p) -- (f) --(h);
		\fill [blue!10] (i) -- (m) -- (q) -- (f) --(i);
		\fill [blue!10] (k) -- (o) -- (r) -- (g) --(k);
		\fill [blue!10] (j) -- (n) -- (s) -- (g) --(j);
		
		\addverts{{a,b,c,d}}{blue}{1};
		\addverts{{M}}{red}{2};
		\addverts{{f,g}}{green}{3};
		\addverts{{h,i,j,k}}{yellow}{4};
		\addverts{{l,m,n,o}}{purple}{5};
		\addverts{{p,q,r,s}}{cyan}{6};
		\end{tikzpicture}
        \caption{Construction for bipartite graphs}\label{figBipart}
	\end{figure}
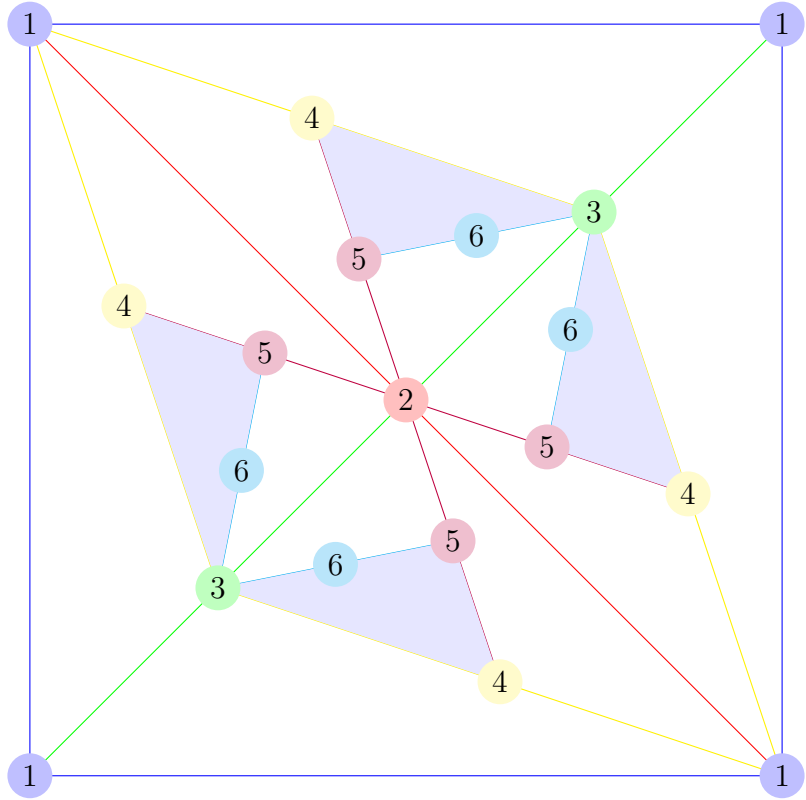
    Let $S_0$ denote the graph in Figure \ref{figBipart}. We will construct many copies of $S_0$. The $4$-cycle constructed in \stepref{bipartinitialcycle} will be the vertices in $S_0$ with label 1. To continue construction of $S_0$, use edges from $E_2$ to insert vertices labeled 2 and 3 (i.e. insert $3$ new vertices along with the edges in Figure \ref{figBipart} between vertices labeled with 1 and 3, labeled with 2 and 3, and labeled with 1 and 2), use edges from $E_3$ to insert vertices labeled 4 and 5, and use edges from $E_4$ to insert the vertices labeled 6. This procedure is repeated anytime we want to construct $S_0$.
    
    Now recursively insert $S_0$ into each of the blue faces (the four shaded faces in Figure \ref{figBipart} each enclosed by vertices with labels 3, 4, 5, and 6), where in the $k$'th recursive call you insert $4^k$ copies of $S_0$, one into each blue face. In addition, if k is odd, use edge sets $E_4,E_5,E_6$, instead of $E_2,E_3,E_4$. Fix a constant $\delta$, $0<\delta < 1/4$, and repeat $m$ recursive calls so that at least $\delta n $ vertices have been used. Note that $m = \Theta(\log n)$. Call the resulting structure $S_m$.
\item \label{bipartfinalinsert} 
	Insert remaining vertices into disjoint faces, one vertex per face, by connecting a pair of antipodal vertices of the face with a path of length $2$ using edges from $E_8$.
\end{enumerate}
Note that $p' > p/8$, that the above construction is bipartite, and furthermore every face has size $4$. By Claim \ref{sizeofmaxplanarG} it suffices to show that the above construction can be done in $G_{n,p}$ w.h.p. For \stepref{bipartinitialcycle}, let $X$ denote the number of 4-cycles in $G_{n,p}$. Then 
$$\Ex[X] = \Omega\left(n^2 \log^2 n\right),
\mathrm{Var}(X) = O \left(n^{5/2} \log^{7/2} n \right).
$$
$$
\implies \mathbf{Pr}(X = 0) = O \left(\frac{1}{n^{3/2}}\right).
$$
For \stepref{bipartfinalinsert}, assume \stepref{bipartmainconstr} is possible. By construction, $S_m$ has maximum degree 8 and at least $\delta n$ vertices. Therefore, there is a $\delta' > 0$ such that we may choose a set of $\delta' n $ vertex disjoint faces. Moreover, taking one of the disjoint faces and connecting two of its antipodal vertices with a $P_2$ does not reduce the number of disjoint faces in $S_m$. Therefore we may attempt to insert the remaining vertices one at a time, each vertex having at least $\delta'n$ faces to be possibly inserted into. The probability that at least one of the remaining vertices can't be inserted into the structure is at most
$$
(1-\delta)n(1 - p'^2)^{\delta'n} \leq (1-\delta)n\left(1 - \frac{p}{64}^2\right)^{\delta' n }  <  n \left( 1 - \frac{C^2 \log n}{64n}\right)^ {\delta' n }< \frac{1}{n}
$$
for $C$ a large enough constant.

Thus it suffices to show that\stepref{bipartmainconstr} can be completed w.h.p. Note that each recursive call has 5 insertion steps, where the ith step inserts the i-labeled vertices. Denote that the set of faces which we are trying to insert a vertex to in the ith step of the kth recursive call as $F_{(k,i)}.$ Denote the set of vertices which we may use to add into the faces of $F_{(k,i)}$ as $N_{(k,i)}$. By our choice of edges used for insertion, each vertex insertion is unconditioned by previous vertex insertions. However, if in a step we simply use edges from $E_j$ to insert vertices into the insertion faces, two insertions in that same step may not be independent of each other (i.e. inserting two vertices into two adjacent faces). To get around this, note that each insertion face in $F_{(k,i)}$ is adjacent to at most three other insertion faces in $F_{(k,i)}$. Hence we may partition $F_{(k,i)} = \bigcup_{l = 1}^4 F_{(k,i,l)}$ where all the faces in $F_{(k,i,l)}$ are vertex disjoint. We may then partition $E_j$ into 4 random edge sets $\{E_{j,1},E_{j,2},E_{j,3},E_{j,4} \}$, where $E_{j,k}$ is distributed as $G_{n,p''}, ~ (1-p'')^4 = (1-p)$ and $p'' > p/32.$ We now use the random edges from $E_{j,l}$ to put edges between $F_{k,i,l}$ and $N_{k,i}$.

We create an auxiliary bipartite graph with partite sets $F_{(k,i,l)}$ and $N_{(k,i)}$ where there is an edge between a face in $F_{(k,i,l)}$ and a vertex in $N_{(k,i)}$ if that vertex may inserted in the face (in other words, the edges in the auxiliary graph appear independently with probability $p''^2$). To show that we may insert all necessary i-labeled vertices, it suffices to show that there is a maximal matching between $F_{(k,i,l)}$ and $N_{(k,i)}$ in the auxiliary graph. Since $(p'')^2 \geq \frac{C^2\log n }{32^2n},$ by Lemma \ref{finalmatching} it follows that we can find a maximal matching with at least probability $1 - \frac{1}{n \log n}$ for $C$ a large enough constant. Since there are at most $O(\log n)$ steps, we can complete (3) with probability $1 - \Omega \left(\frac{1}{n}\right)$ Therefore the probability that any of steps (2 - 4) fail is at most $O(\frac{1}{n})$.

\section{Subgraphs of Girth $g = 2k$}
The proof for the upper bound for girth $g = 2k$ for $k \geq 3$ uses a construction motivated by the smaller girth cases but requires a simpler analysis.

Let $G_{n,p}$ be the random graph and let $E$ be its edge set. Let $p = (4M+Q+2)C_2' \left(\frac{\log n}{n^{(g-2)/g}}\right)$, where $M,Q$ are defined below and $C_2'$ is chosen so that Lemma \ref{pathlemma} is satisfied under the assumption that $k=\frac{g}{2}$ and $C_1 = 2$. Let $s = \frac{g-2}{4}$ if $g\equiv 2 \pmod{4}$ and $s = \frac{g}{4}$ if $s\equiv 0 \pmod{4}$. Consider the following construction:
\begin{enumerate}
\item \label{2kedgesplit}
	Define $\bigcup_{i = 1}^{4M + Q + 2} E_i := E$ where $E_i$ is distributed as $G_{n,p'}$ where $(1 - p')^{4M + Q + 2} = 1 - p, p' > \frac{p}{4M + Q + 2}, M = \Theta(\log n)$ and will be defined later and $Q = O(1)$ and will be defined later.
\item \label{2kinitialcycle}
	Use edges from $E_1, E_2$ to construct a $g$-cycle.
\item\label{2kmainconstr1}
	If $g = 4s+2$, consider figure \ref{fig4k} below. Otherwise, consider the next step (\ref{2kmainconstr2}) 
	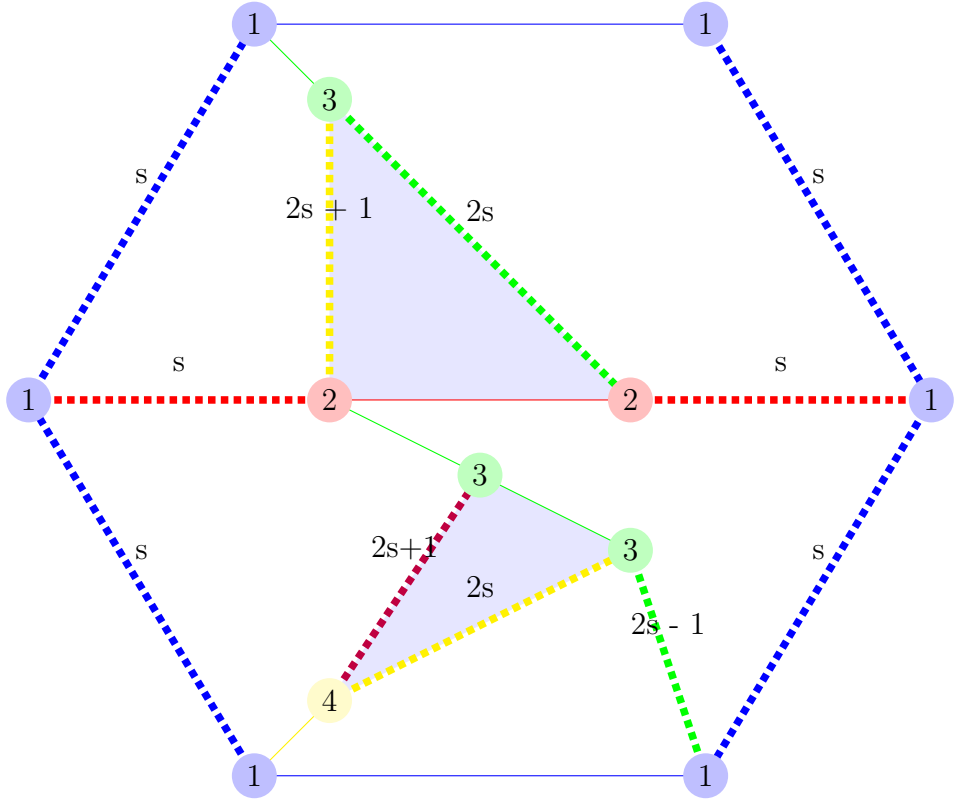
\begin{figure}
		\centering
		\begin{tikzpicture}
			\coordinate (a) at (-3,5);
			\coordinate (b) at (3,5);
			\coordinate (c) at (6,0);
			\coordinate (d) at (3,-5);
			\coordinate (e) at (-3,-5);
			\coordinate (f) at (-6,0);
			
			\coordinate (g) at (2,0);
			\coordinate (h) at (-2,0);

			\coordinate (i) at (0,-1);
			\coordinate (j) at (-2,4);
			\coordinate (k) at (2,-2);
			
			\coordinate (l) at (-2,-4);

			\fill [blue!10] (h) -- (g) -- (j) --(h);
			\fill [blue!10] (l) -- (i) -- (k) -- (l);
			
			\draw[yellow] (e) -- (l);
			\addpath{l}{i}{2s+1}{purple};
			\addpath{l}{k}{2s}{yellow};
			\addpath{h}{j}{2s + 1}{yellow};
			
			\draw[green] (a) -- (j);
			\addpath{j}{g}{2s}{green};
			\draw[green] (h) -- (i);
			\draw[green] (i) -- (k);
			\addpath{k}{d}{2s - 1}{green};
			
			\draw[red] (g) -- (h);
			\addpath{c}{g}{s}{red};
			\addpath{f}{h}{s}{red};
			
			\draw[blue] (a) --(b);
			\addpath{b}{c}{s}{blue};
			\addpath{c}{d}{s}{blue};
			\draw[blue] (d) -- (e);
			\addpath{e}{f}{s}{blue};
			\addpath{f}{a}{s}{blue};
			
			\addverts{{a,b,c,d,e,f}}{blue}{1};
			\addverts{{g,h}}{red}{2};
			\addverts{{i,j,k}}{green}{3};
			\addverts{{l}}{yellow}{4};
		\end{tikzpicture}
        \caption{Construction for girth $4s + 2$}\label{fig4k}
	\end{figure}
Denote the structure it represents, where the blue faces (the two shaded triangles, one between vertices labeled 2-2-3 and one between vertices labeled 3-3-4) are empty, as $H_0$. Use the edges from $E_3$ to construct the red path of length $2s+1$ lebeling the middle 2 vertices with label 2, use the edges from $E_4$ to construct the pair of green paths of length $2s+1$ between vertices labeled 2 and 1 with vertices labeled 3 on them, use the edges from $E_5$ to construct the yellow paths of length $2s+1$ (one between vertices labeled 2 and 3 and one between vertices labeled 1 and 3 putting a vertex with label 4 on the path), and use the edges from $E_6$ to construct the burgundy path (of length $2s+1$ between vertices labeled 3 and 4). Now recursively insert $H_0$ into each of the blue faces, where in the $k$'th recursive call you insert $2^k$ copies of $S_0$, one into each blue face from the $k-1$'th recursive call. We assume the $0$th recursive means constructing $H_0$ from the $g$-cycle. In addition, for the $k$'th recursive call, use edge sets $E_{4k + 3},E_{4k + 4},E_{4k+5}\text{ and }E_{4k+6}.$ Do this for $M$ recursive calls and call the resulting structure $H_M$. Choose $M$ so that $H_M$ has $\delta n $ vertices where $\delta$ is a fixed constant $0<\delta < 1/2$. Note that $M = \Theta(\log n)$.
\item \label{2kmainconstr2}
	Consider Figure \ref{fig4k+2} below.
	
    \begin{figure}
    	\centering
		\begin{tikzpicture}
			\coordinate (a) at (-3,6);
			\coordinate (b) at (3,6);
			\coordinate (c) at (6,3);
			\coordinate (d) at (6,-3);
			\coordinate (e) at (3,-6);
			\coordinate (f) at (-3,-6);
			\coordinate (g) at (-6,-3);
			\coordinate (h) at (-6,3);
			
			\coordinate (i) at (0,0);
			\coordinate (j) at (2,1);
			\coordinate (k) at (-2,-1);
			
			\coordinate (l) at (-1,1);
			\coordinate (m) at (1,3);
			
			\coordinate (n) at (-2,-5);
			\coordinate (o) at (-1,-4);
			
			\coordinate (p) at (-2,5);

			\fill [blue!10] (p) -- (m) -- (l) --(p);
			\fill [blue!10] (o) -- (j) -- (i) -- (k) -- (o);
			
			\addpath{o}{k}{2s}{purple};
			\addpath{p}{m}{2s}{purple};
			
			\draw[yellow] (a)--(p);
			\addpath{p}{l}{2s - 1}{yellow};
			
			\draw[green] (k)--(l);
			\draw[green] (l)--(m);
			\addpath{m}{b}{2s - 2}{green};
			
			\draw[green] (f)--(n);
			\draw[green] (n)--(o);
			\addpath{o}{j}{2s - 2}{green};
			\draw[red] (i)--(j);
			\draw[red] (i)--(k);
			\addpath{j}{c}{s-1}{red};
			\addpath{k}{g}{s-1}{red};
			
			\draw[blue] (a)--(b);
			\addpath{b}{c}{s - 1}{blue};
			\draw[blue] (c)--(d);
			\addpath{d}{e}{s - 1}{blue};
			\draw[blue] (e)--(f);
			\addpath{f}{g}{s - 1}{blue};
			\draw[blue] (g)--(h);
			\addpath{h}{a}{s - 1}{blue};
			
			\addverts{{a,b,c,d,e,f,g,h}}{blue}{1};
			\addverts{{i,j,k}}{red}{2};
			\addverts{{l,m,n,o}}{green}{3};
			\addverts{{p}}{yellow}{4};
		\end{tikzpicture}
        \caption{Construction for girth $4s$}\label{fig4k+2}
	\end{figure}
    Denote the structure it represents, where the blue (shaded) faces are empty, as $H_0$. Use the edges from $E_3$ to construct the red path of length $2s$ between a pair of vertices labeled 1 putting three vertices of label 2 n the middle of the path, use the edges from $E_4$ to construct the green paths of length $2s$ between vertices labeled 1 and 2 putting a pair of vertices with label 3 on each path, use the edges from $E_5$ to construct the yellow path of length $2s$ between vertices labeled 1 and 3 putting a vertex of label 4 on the path, and use the edges from $E_6$ to construct the burgundy paths of length $2s$ (one between vertices labeled 3 and 4 and one between vertices labeled 2 and 3). Now recursively insert $H_0$ into each of the blue faces, where in the $k$'th recursive call you insert $2^k$ copies of $S_0$, one into each blue face created during the $k-1$'st recursive call. We assume the $0$'th recursive call means constructing $H_0$ from the $g$-cycle. In addition, for the $k$'th recursive call, use edge sets $E_{4k + 3},E_{4k + 4},E_{4k+5}\text{ and }E_{4k+6}$ respectively. Do this for $M$ recursive calls and denote the resulting structure as $H_{M}$. Choose $M$ so that $H_M$ has $\delta n $ vertices where $\delta$ is a fixed constant $0<\delta<1/2$. Note that $M = \Theta(\log n)$.
\item \label{2kfinalinsert} 
	Find a set of $\epsilon n$ vertex disjoint faces where $\epsilon> 0$ is a constant. In each face, fix a pair of antipodal vertices. Then use the remaining vertices not in $H_M$ to insert paths of length $g/2$ between each pair of antipodal vertices (splitting a face into a pair of cycles of length $g$). Continue this process until all vertices have been inserted.
\end{enumerate}

Since the construction satisfies the condition of Claim \ref{sizeofmaxplanarG} it is indeed a maximal planar graph of girth $g$.
Thus it suffices to show that the construction can be completed w.h.p. As we did before, we will show that the $p$ is sufficient. However, we will calculate the total failure probability at the end. Take a vertex set $S$ of size $\epsilon n(\frac{g}{2} + 1)$ where $0<\epsilon<\frac{1}{2g}$ is a constant. Take $\epsilon n$ vertices from $S$ and label them $\{a_i\}$ and take another $\epsilon n$ vertices from $S$ and label them $\{b_i\}$. Note that $E_1$ restricted to $S$ is distributed as $G_{\epsilon n(\frac{g}{2} + 1),p'}.$ Therefore, by Lemma \ref{pathlemma} and our choice of $p,$ w.h.p. $E_1$ contains disjoint paths between each $(a_i,b_i)$ pair of length $\frac{g}{2}.$ Let $v_1,\cdots,v_{\frac{g}{2} - 1}$ be the interior vertices in the path from $a_1$ to $b_1$. Let $S'$ be the vertex set obtained by taking $S$ and replacing $v_1,\cdots,v_{\frac{g}{2} - 1}$ with $\frac{g}{2} - 1$ vertices not already in $S.$ Since $E_2$ restricted to $S'$ is distributed as $G_{\epsilon n(\frac{g}{2} + 1),p'}$ by Lemma \ref{pathlemma} w.h.p. $E_2$ contains a path between $a_1$ and $b_1$ of length $\frac{g}{2}$. Since $v_1,\cdots,v_{\frac{g}{2} - 1} \not \in S'$, the interior vertices of both paths are disjoint. Therefore, we can append the paths to form a cycle of length $g$, and so with probability at least $1-\Omega\left(n^{-2}\right)$, $E_1,E_2$ contain a cycle of length $g$. 
\newline 

We now consider \stepref{2kmainconstr1} and \stepref{2kmainconstr2}. Suppose we are in the $i$th step of the $k$th recursive call. Let $\{(a_1,b_1),\cdots (a_m,b_m)\}$ be the collection of vertex pairs which we want to connect with paths of length $g/2$ using edges from $E_{4k + i + 2}.$ Since $M$ is chosen such that $H_M$ has $\delta n$ vertices with $\delta<1$, we can take a subset of the unused vertices and the vertices in the $(a_i,b_i)$ pairs to get a set of vertices $S$ of size $\frac{n}{2g}(\frac{g}{2} + 1)$. Now take subsets $A, B\subset S$ of size $\frac{n}{2g}$ such that $a_i\in A$ and $b_i\in B$.  Since $E_{4k + i+2}$ restricted to $S$ is distributed as $G_{\frac{n}{2g}(\frac{g}{2} + 1),p'},$ by our choice of $p$ and Lemma \ref{pathlemma} w.h.p. we can find disjoint paths connecting the $(a_i,b_i)$ vertex pairs with paths of length $\frac{g}{2}$ using edges from $E_{4k + i+2}.$

We now consider \stepref{2kfinalinsert}. Note that in both constructions, a face can be incident to at most $13$ other faces. Since $H_M$ has $\delta n$ vertices and each face is incident with a bounded number of faces, we may find a set of $\epsilon n$ vertex disjoint faces where $\epsilon > 0$ is a constant. Call these faces $F_1,\cdots, F_{\epsilon n}$. For each face, choose a pair of antipodal vertices $a_i, b_i \in F_i$. Partition the remaining vertices into sets of size at least $\frac{1}{2}\left(\frac{g}{2} -1\right)\epsilon n$ and at most $\left(\frac{g}{2}-1\right)\epsilon n$ and call these sets $S_1,\cdots , S_Q$. Also assume that the size of each $S_i$ is divisible by $\frac{g}{2}-1$ (this is possible because of the divisibility condition on $n$). Note that $Q = O(1)$. Now for each $i$, $1\leq i\leq R$, assume that $|S_i| = C_i\left(\frac{g}{2}-1\right)$. We wish to insert the vertices from $S_i$ into $F_1, F_2,\cdots F_{C_i}$. Since $E_{4M+2+i}$ restricted to $S_i$ is distributed as $G_{C_i(g/2+1), p'}$, Lemma \ref{pathlemma} implies that there are disjoint paths of length $\frac{g}{2}$ between the vertex pairs $(a_1,b_1),\cdots,(a_{C_i}, b_{C_i})$ with probability at least $1 - \frac{1}{((\epsilon/2)n)^2}$.


We now compute the total failure probability. Note that their are at most $4M + Q + 2$ vertex insertion steps. By our choice of $p$, each insertion step succeeds with probability $1 - \Omega\left(n^{-2}\right)$. Therefore the entire insertion process exceeds with probability tending to $1$.

\section{Subgraphs of Girth $g = 2k + 1$}
For the proof of the odd girth case, we again provide a construction. Let $g = 2k+1$. Let $p = C_2\left(\frac{\log n}{n^k}\right)^{\frac{1}{k+1}}$ where $C_2$ is chosen to be a large enough constant that we can apply Lemma \ref{pathlemma} with $C_1=2$ throughout the proof.
\begin{enumerate}
	\item \label{oddedgesplit} 
    	Define $M$ such that $Mg + (M-1)g(g-3) = (1-\delta)n$ where $\delta$ is a fixed constant with $0<\delta < \frac{1}{5g}$. Define $\cup_{i = 1}^{8} E_i = E$, where an edge in $E$ appears in $E_i$ with probability $p'$, where $(1 - p')^8 = 1-p$ and $p' > \frac{p}{8}$.
    \item \label{oddcycleconstr} 
    	Construct $M$ cycles of length $g$ using edges from $E_1$ and $E_2$.
    \item \label{oddcyclefill}
    	Consider the figure below:
  \begin{figure}
      \centering
      \begin{tikzpicture}
          \def\p{6};
          \def\n{15};
          \foreach \rad in {7,5,3}{
              \foreach \i in {1,...,\p}{
                  \draw[color=blue] (\i*360/\n - 360/\n:\rad) -- (\i*360/\n:\rad);
              }
              \draw[dotted,line width = 1mm,color=blue] (\p*360/\n:\rad) arc (\p*360/\n:360:\rad)
                  node[midway,below,outer sep = 7pt,text = black] {$k$};
          }
          \foreach \rad in {5,3}{
              \foreach \i in {0,...,\p}{
                   \draw[dotted,line width = 1mm,color=red] (\i*360/\n:\rad) -- (\i*360/\n:\rad + 2)
                   		node[midway,outer sep = 7pt,text=black] {$k$};
              }
          }
          \foreach \rad in {7,5}{
              \foreach \i in {1,...,\p}{
                  \draw[dotted,line width = 1mm,color=green] (\i*360/\n:\rad) -- (\i*360/\n - 360/\n:\rad - 2)
                  		node[midway,outer sep = 7pt,text=black] {$k$};
              }
          }  
          \foreach \rad in {7,5,3}{
              \foreach \i in {0,...,\p}{
                  \node[circle,fill=blue!25,minimum size=17pt,inner sep=0pt] at (\i*360/\n:\rad) {$1$};
              }
          }
          \foreach \i in {0,...,\p}{
          	\foreach \rad in {2.5,2,1.5}{
            	\draw[fill=black,opacity = \rad/2] (\i*360/\n:\rad) circle (1pt);
            }
          }
          \foreach \rad in {4,6}{
          	\foreach \i in {1,...,8}{
            	\draw[fill=black,opacity = 1/\i] (\p*360/\n + \i*90/\n:\rad) circle (1pt);
            }
          }
      \end{tikzpicture}
      \caption{Construction for girth $g = 2k + 1$}\label{figOddG}
  \end{figure}
  We perform this construction as follows:
    \begin{enumerate}
    \item Enumerate the cycles from $1$ to $M$ as $C_i$.
    \item For each cycle $C_i,$ enumerate it's vertices as $a_{1,i},a_{2,i},\cdots, a_{g,i}$ in clockwise order.  
    \item For each cycle pairing $C_{2s - 1},C_{2s}$ simultaneously add a disjoint path of length $k$ between $a_{j,2s-1}$ and $a_{j,2s}$ for each $1 \leq i \leq g$ using edges from $E_3$. Then, simultaneously add a disjoint path of length $k$ between $a_{j,2s-1}$ and $a_{j+1,2s}$ for each $1 \leq i \leq g$, were $g+1 := 1,$ using edges from $E_4.$  
    \item Apply (c) again, but this time between all cycle pairings $C_{2s,g}, C_{2s+1,g}$ and using edges from $E_5$ and $E_6$ resp.
    \end{enumerate}
  \item \label{oddfinalinsert}
  Assume that there are $C(g-1)$ vertices remaining. Note that $C$ is an integer by the divisibility condition on $n$ and $C = \Omega(n)$ by the choice of $M$. We may find at least $\lfloor M/2 \rfloor$ vertex disjoint faces in the construction at this point. By choice of $M$, $M/2 > C$ and so we may fix disjoint faces $F_1,\cdots F_c$. We will insert the remaining vertices into these faces, $g$ vertices per face as in the figure below, using all of the remaining vertices. 
    \begin{figure}
      \centering
      \begin{tikzpicture}
          \def\n{15};
          \def\s{7};
          \def\rad{5};
          \draw[dotted,line width = 1mm,color=blue] (90:\rad) arc (90:\s*360/\n + 90:\rad)
          		node[midway,left,outer sep = 7pt,text=black] {$k$};
          \draw[dotted,line width = 1mm,color=blue] (90:\rad) arc (90:90 - \s*360/\n:\rad)
         		node[midway,right,outer sep = 7pt,text=black] {$k$};
          \draw[color=blue] (\s*360/\n + 90:\rad) -- (90 - \s*360/\n:\rad);
          
          \fill[color=blue!10] (90:\rad - 1) -- (90 - \s*360/\n:\rad) -- (90 + \s*360/\n:\rad) -- (90:\rad - 1);
          
          \draw[color=red] (90:\rad) -- (90:\rad - 1);
          \draw[dotted,line width = 1mm,color = red] (90:\rad - 1) -- (90 - \s*360/\n:\rad)
          		node[midway,right,outer sep = 7pt,text=black] {$k$};
          \draw[dotted,line width = 1mm,color = green] (90:\rad - 1) -- (90 + \s*360/\n:\rad)
          		node[midway,left,outer sep = 7pt,text=black] {$k$};
          
          \node[circle,fill=blue!25,minimum size=17pt,inner sep=0pt] at (90:\rad) {$1$};
          \node[circle,fill=blue!25,minimum size=17pt,inner sep=0pt] at (90 + \s*360/\n:\rad) {$1$};
          \node[circle,fill=blue!25,minimum size=17pt,inner sep=0pt] at (90 - \s*360/\n:\rad) {$1$};
          \node[circle,fill=red!25,minimum size=17pt,inner sep=0pt] at (90:\rad - 1) {$2$};
      \end{tikzpicture}
      \caption{Insertion process for girth $g = 2k + 1$}\label{insertOdd}
    \end{figure}
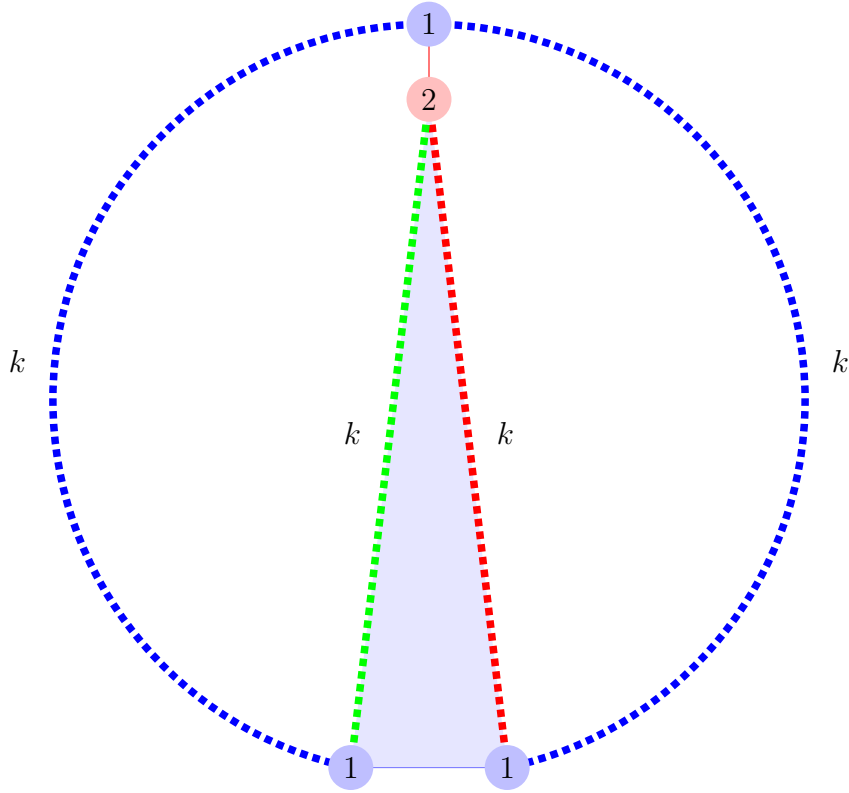
     
\end{enumerate}
We claim that this construction if maximal planar. To see this, note that after \stepref{oddcyclefill} of the construction, it's clear that each face of the current graph is a cycle of length $g$. In \stepref{oddfinalinsert}, we note that in each vertex insertion step, a face is replaced with 3 faces, each a cycle of length g. 

We now elaborate on the construction. For \stepref{oddcycleconstr}, We take $\{(a_1,b_1), \cdots (a_M,b_M)\}$ to be disjoint vertex pairs from the vertex pool $V$. Let $S_1$ and $S_2$ be disjoint sets of vertices of size $(k-1)M$ and $kM$ respectively and each disjoint from $\{a_i\} \cup \{b_i\}$. By Lemma \ref{pathlemma} we may find disjoint paths of length $k$ connecting the pairs $(a_i,b_i)$ with the internal vertices coming from $S_1$ and the edges coming from $E_1$ w.h.p. Similarly, we may find disjoint paths of length $k+1$ connecting the pairs with the internal vertices coming from $S_2$ and the edges coming from $E_2$ w.h.p. Since $S_1$ and $S_2$ are disjoint, we have formed $M$ disjoint cycles each of length $g$ with probability at least $1-\Omega\left(n^{-2}\right)$. 

 \medskip
 
 For \stepref{oddcyclefill}, we can again build large sets of disjoint paths of length $k$, were we build these paths between the following vertex pairs: 
\begin{align*}
	&\bigcup_{\substack{1 \leq i \leq M \land 2 \nmid i\\ 1 \leq k \leq g}} (a_{k,i},a_{k,i+1})\\
    &\bigcup_{\substack{1 \leq i \leq M \land 2 \nmid i\\ 1 \leq k \leq g}} (a_{k,i},a_{k+1,i+1})\\
    &\bigcup_{\substack{1 \leq i \leq M \land 2 \mid i\\ 1 \leq k \leq g}} (a_{k,i},a_{k,i+1})\\
    &\bigcup_{\substack{1 \leq i \leq M \land 2 \mid i\\ 1 \leq k \leq g}} (a_{k,i},a_{k+1,i+1})
\end{align*}

where each set of paths uses edges from $E_3,E_4,E_5$ and $E_6$ respectively. By our choice of $M$, we may always find enough unused vertices to create these paths and Lemma \ref{pathlemma} can be applied with our choice of $p$ and we may find these paths with probability at least $1-\Omega(1/n^2)$. 

\medskip

For \stepref{oddfinalinsert}, we note that one may choose a face arbitrarily between $C_i$ and $C_{i+1}$ to find $\lfloor{M/2} \rfloor$ vertex disjoint faces $F_1, \cdots F_{M/2}$. Assume that there are $C(g-1)$ vertices remaining, and note that $C < M/2$ by the choice of $M$. For $F_1,\cdots F_C$ choose a vertex in $F_i$ arbitrarily and call it $a_i$. Let $b_i$ and $c_i$ be the two vertices antipodal to $a_i$. Let $S$ be a set of $kC$ of the remaining vertices. By Lemma \ref{pathlemma}, we may find paths of length $k+1$ between $a_i$ and $b_i$ with the internal vertices coming from $S$ with probability at least $1 - \Omega(1/C^2) = 1 - \Omega(1/n^2)$. For each $i$ let $a_i'$ be the vertex on these paths that is adjacent to $a_i$. There are now exactly $C(k-1)$ vertices remaining. Applying Lemma \ref{pathlemma} again allows us to connect each $a_i'$ to $c_i$ with a path of length $k$ where the internal vertices come from the remaining vertices not used with probability at least $1 - \Omega(1/n^2)$.

Since we have applied Lemma \ref{pathlemma} a constant number of times the construction can be found with probability tending to $1$.

\section{Conclusion}
Although we were able to find intervals for the even girth cases that are optimal within a poly-log factor, we were unable to do so for the odd girth case and we leave this as an open question. It would also be very interesting to remove the logarithmic factor entirely in any of the cases.

\section{Acknowledgements}
The first two authors would like to thank Alan Frieze for his course on random graphs and all of the authors would like to thank him for helpful discussions.
\bibliographystyle{alpha}
\bibliography{subgraphbib}	

\end{document}